\documentclass[12pt]{amsart}
\usepackage{epsfig,psfrag,verbatim,amssymb,amsmath,amsfonts}
\newtheorem{theorem}{Theorem}
\newtheorem{lemma}[theorem]{Lemma}

\newtheorem{theo}{Theorem}
\newtheorem{pro}[theo]{Proposition}
\newtheorem{ex}{Example}

\newcommand{\EE}{\mathbb{E}}
\newcommand{\R}{\mathbb{R}}
\newcommand{\Q}{\mathbb{Q}}
\newcommand{\Z}{\mathbb{Z}}
\newcommand{\N}{\mathbb{N}}

\newcommand{\PP}{\mathbb{P}}
\newcommand{\la} {\lambda}

\newcommand{\al}{\alpha}
\newcommand{\ga}{\gamma}

\newcommand{\La}{\Lambda}
\newcommand{\eps}{\varepsilon}
\newcommand{\won}{{\boldsymbol 1}}

\newcounter{constante}
\newcommand{\con}[1]{
\immediate\write 1{\noexpand\newlabel{#1}{{\theconstante}{\theconstante}}}
c_{\theconstante}
\stepcounter{constante}
}

\begin{document}

\setcounter{page}{1}

\title[Lyapunov exponents] {
Lyapunov exponents of Green's functions for
random potentials tending to zero}

\author{Elena Kosygina, Thomas S.\ Mountford, Martin P.W.\ Zerner} 

\thanks{\textit{2000
Mathematics Subject Classification.} 60K37, 82B41, 82B44. }
\thanks{\textit{Key words:}\quad Annealed, Green's function, Lyapunov exponent, quenched, random potential, random walk. }

\begin{abstract}
We consider quenched and annealed Lyapunov exponents for the Green's function of $-\Delta+\ga V$, where the potentials $V(x),\ x\in\Z^d$, are i.i.d.\ nonnegative random variables and  $\ga>0$ is a scalar. We present a probabilistic proof that both Lyapunov exponents scale like $c\sqrt{\ga}$ as $\ga$ tends to 0. Here the constant $c$ is the same for the quenched as for the annealed exponent and is computed explicitly.  This improves results obtained previously by Wei-Min Wang. We also consider other ways to send the potential to zero than multiplying it by a small number.
\end{abstract}
\maketitle

%%%%%%%%%%%%%%%%%%%%%%%%%%%%%%%%%%%%%%%%%%%%%%%
\section{Introduction, results, and examples}
%%%%%%%%%%%%%%%%%%%%%%%%%%%%%%%%%%%%%%%%%%%%%%%
We consider the symmetric,  nearest-neighbor random walk $(S(n))_{n\ge 0}$ in discrete time on $\Z^d,\ d\ge 1,$ which starts at 0. The probability measure and expectation operator of the underlying probability space are denoted by $P$ and $E$ respectively. The random walk evolves in a random potential $V=(V(x))_{x\in\Z^d}$ consisting of i.i.d.\ non-negative random variables $V(x), x\in\Z^d,$ which are defined on a different  probability space with probability measure $\PP$ and expectation operator $\EE$.
To avoid trivialities we assume that $\PP[V(0)>0]>0$.
Given a potential $V$ and $y\in\Z^d$ we define the random walk's Green's function of $0$ and $y$ as
\begin{equation}\label{gr}
g(0,y,V):=\sum_{m\ge 0}E\left[\exp\left(-\sum_{n=0}^{m}V(S(n))\right)\cdot\won_{S(m)=y}\right].
\end{equation}
This function has the following well-known interpretation, see e.g.\ \cite[pp.\ 249]{Ze98}:
If each visit to a vertex $x$ with potential $V(x)$ ``kills'' the walk with probability $1-e^{-V(x)}$ then $g(0,y,V)$ is the expected number of visits of the random walk to $y$ before the walk is killed. 

Closely related to the random walk's Green's function $g$ is the operator's Green's function $G$ of $-\Delta+V$ which is defined as the unique bounded solution of 
\[(-\Delta+V)G(0,y,V)=\delta_{0,y},\]
where the discrete Laplacian is given by $\Delta f(y):=\left(\sum_{|e|=1} f(y+e) \right) / (2d) - f(y)$. In fact, there is a one-to-one correspondence between these two functions, namely
\begin{equation}\label{UV}
G(0,y,V)=g(0,y,\ln (V+1)),
\end{equation}
see e.g.\ \cite[Proposition 2]{Ze98}.  
Consequently, it suffices to study either $g$ or $G$. We choose to study $g$
and are interested in the exponential rate of decay of $g(0,y,V)$ as $|y|\to\infty$. This was  investigated in 
\cite{Ze98}. There the  function  
\begin{eqnarray*}
e(0,y,V)&:=&E\left[\exp\left(-\sum_{n=0}^{H(y)-1}V(S(n))\right)\cdot\won_{H(y)<\infty}\right]
\end{eqnarray*}
was introduced, where $H(y):=\inf\{n\ge 0\mid S(n)=y\}$ is the first passage time of the random walk through $y$. The quantity $e(0,y,V)$ can be interpreted as the  probability that the random walk reaches $y$ before being killed. For the most part the following result is contained in \cite{Ze98}. We shall comment on it in the appendix. %%%%%%%%%%%%%%
\begin{pro}  \label{zp}
Assume $\EE[V(0)]<\infty$.  Then 
there is a non-random norm
$\al_V$ on $\R^d$, the so-called  quenched Lyapunov exponent, such that $\PP$-a.s.\ for all $\ell\in\Z^d$,
\begin{eqnarray}\label{do}
\al_V(\ell)
&=&{\displaystyle \lim_{k\to\infty}\frac{-\ln e(0,k\ell,V)}{k}}
\ =\ {\displaystyle \lim_{k\to\infty}\frac{-\EE\left[\ln e(0,k\ell,V)\right]}{k}}\\
&=&{\displaystyle \lim_{k\to\infty}\frac{-\ln g(0,k\ell,V)}{k}}
\ =\ {\displaystyle \lim_{k\to\infty}\frac{-\EE\left[\ln g(0,k\ell,V)\right]}{k}.}
\label{do2}
\end{eqnarray}
The norm $\al_V$ is invariant under the isometries of $\Z^d$ which preserve 0. Moreover, if  the potential $V$ is more variable than another i.i.d.\ potential $W=(W(x))_{x\in\Z^d}$, i.e.\ if $E[h(V(0))]\le E[h(W(0))]$ for all increasing and concave functions $h:\R\to\R$, then $\al_V\le \al_W$.
\end{pro}
%%%%%%%%%%%%%%%
A more accurate notation than $\al_V$ would be $\al_{\PP_{V(0)}}$ since the norm does not depend on the whole field $(V(x))_{x\in\Z^d}$ but only on the distribution $\PP_{V(0)}$ of $V(0)$. However, for simplicity we shall use the notation $\al_V$. By Proposition \ref{zp} and (\ref{UV}), 
\begin{equation}
\label{sim}A_V(\ell):=\lim_{k\to\infty}\frac{-\ln G(0,k\ell,V)}{k}=\lim_{k\to\infty}\frac{-\EE\left[\ln G(0,k\ell,V)\right]}{k}
\end{equation}
 exists as well whenever $\EE[\ln(V(0)+1)]<\infty$ and is related to $\al_V$ through
 \begin{equation}\label{UV2}
 A_V=\al_{\ln(V+1)}.
 \end{equation}
By first averaging the function $e(0,k\ell,V)$ with respect to $\PP$ and then taking the logarithm in the definition of the quenched Lyapunov exponents one obtains the so-called annealed or averaged Lyapunov exponents. 
The following result is partially contained in \cite{Fl07}. 
We shall comment on it in the appendix.
%%%%%%%%%%%%%%%
\begin{pro} \label{db}
There is a non-random norm
$\beta_V$ on $\R^d$, the so-called annealed or averaged Lyapunov exponent, such that for all $\ell\in\Z^d$,
\begin{equation}\label{be}
\beta_V(\ell)=\lim_{k\to\infty}\frac{-1}{k}\ln \EE\left[e(0,k\ell,V)\right]=\lim_{k\to\infty}\frac{-1}{k}\ln \EE\left[g(0,k\ell,V)\right].
\end{equation}
The norm $\beta_V$ is invariant under the isometries of $\Z^d$ which preserve 0.
\end{pro}
%%%%%%%%%%%%%%%
Note that while the quenched Lyapunov exponent $\al_V$ has not been defined in Proposition \ref{zp} if $\EE[V(0)]=\infty$,   Proposition \ref{db} states that the annealed exponent $\beta_V$ is well-defined and finite even in this case.

Similarly to (\ref{sim}) and (\ref{UV2}) we set
\begin{equation}\label{sim2}
B_V(\ell):=\lim_{k\to\infty}\frac{-\ln \EE[G(0,k\ell,V)]}{k}=\beta_{\ln(V+1)}(\ell).
\end{equation}
It follows immediately from Jensen's inequality that 
\begin{equation}\label{jen}
\beta_V\le \al_V\quad\mbox{and}\quad B_V\le A_V.
\end{equation}
We refer the reader to the book \cite{Sz98} for detailed information and literature about Lyapunov exponents and related quantities for Brownian motion among Poissonian obstacles. Chapter 7 of this book also discusses connections with other models and gives an overall account of then known results and open problems.  One open problem, which is mentioned in \cite[p.\ 326]{Sz98}, is to prove the equality of quenched and averaged Lyapunov exponents for small potentials in high dimensions.  For a class of random walks with drift this problem was solved in \cite{Fl08} for $d\ge 4$. For the simple symmetric random walk, it was proved in \cite{Zy09} that if $d\ge 4$ then for each $\la>0$ there is  $\gamma^*(\lambda)>0$ such that $\alpha_{\lambda+\gamma V}=\beta_{\lambda+\gamma V}$ for all $\gamma\in[0,\gamma^*(\la))$. The question whether this also holds for $\la=0$ is still open.

In the present paper we consider the behavior of the quenched and the annealed  Lyapunov exponents as the potential tends to zero and show that asymptotically they behave in the same way. This question was previously investigated in \cite{Wa01} and \cite{Wa02}. These papers study the asymptotic behavior of $A_{\ga V}$ and $B_{\ga V}$ as $\ga\searrow 0$, where $\ga>0$ is a scalar.
%%%%%%%%%%%%%%%
\begin{theo}{\rm (\cite[Theorem 4.2, Corollary]{Wa01} and \cite[Theorem 4.3]{Wa02})}
\label{wang}
Assume $\EE[V(0)^2]<\infty$ and let
 $\|\ell\|_2=1$. Then there is a constant $c>0$ which depends only on $d$ and $\EE[V(0)]$ such that  
\[
c< \liminf_{\ga\searrow 0}\frac{B_{\ga V}(\ell)}{\sqrt{\ga}}\le \limsup_{\ga\searrow 0}
\frac{A_{\ga V}(\ell)}{\sqrt{\ga}}
<\infty.
\]
\end{theo}
%%%%%%%%%%%%%%
In fact, the main result of \cite{Wa01} is Theorem \ref{wang} under the stronger assumption for the upper bound that the distribution of $V(0)$ has bounded support. In \cite{Wa02} this assumption was weakened to finiteness of  $\EE[V(0)^2]<\infty$ and it was suggested in the remark after \cite[Theorem 4.3]{Wa02}  that this weaker assumption was likely to be optimal for the conclusion of Theorem \ref{wang} to hold. The proofs use a supersymmetric representation of the averaged Green's function and multiscale analysis.

In the present paper, we give a relatively elementary proof of a stronger version of Theorem \ref{wang}.
It shows, in particular, that the statement of Theorem \ref{wang} holds if and only if $\EE[V(0)]$ is finite.
%%%%%%%%%%%%%%%%%%%%%%%
\begin{theorem} 
\label{cor} Assume that $\EE[\ln(V(0)+1)]<\infty$ and 
let $\|\ell\|_2=1$. Then 
\[
\lim_{\ga\searrow 0}\frac{A_{\ga V}(\ell)}{\sqrt{\ga}}=\lim_{\ga\searrow 0}\frac{B_{\ga V}(\ell)}{\sqrt{\ga}}=
\sqrt{2d\, \EE[V(0)]}.
\]
\end{theorem}
%%%%%%%%%%%%%%%%%%%%
In fact, Theorem \ref{cor} follows from a more general result, see Theorem \ref{new} and Example \ref{ex2} below. Note that the common limit in Theorem \ref{cor} is invariant under rotations of $\ell$. It is also invariant under the replacement of the potential $V$ by its mean. The latter property indicates that both Lyapunov exponents exhibit mean field behavior for small potentials.

Although multiplying the potential $V$ by a constant $\ga$ and then letting $\ga$ go to zero is probably the simplest way to send the potential to zero, there are other ways to achieve this, which are covered by our approach as well.  In the following we shall assume that we have a family $(V_\ga)_{\ga>0}$ of i.i.d.\
non-negative potentials $V_\ga=(V_\ga(x))_{x\in\Z^d}$ and obtain upper and lower bounds on the asymptotic behavior of the associated Lyapunov exponents $\al_{V_\ga}$ and $\beta_{V_\ga}$ as $\ga\searrow 0$. 
%%%%%%%%%%%%%%%%
\begin{theorem}\label{upp}Assume that $\EE[V_\ga(0)]<\infty$ for all $\ga>0$. Then for all $\ell\in\R^d$,
\[
\limsup_{\ga\searrow 0}\frac{\al_{V_\ga}(\ell)}{\sqrt{\ga}}\le\limsup_{\ga\searrow 0}\sqrt{\frac{2d\, \EE[V_\ga(0)]}{\ga}}\ 
\|\ell\|_2.
\]
\end{theorem}
%%%%%%%%%%%%%%%%%
%%%%%%%%%%%%%%%%%
\begin{theorem}\label{lowp}
Assume that $V_\ga(0)/\ga$ converges in distribution as $\ga\searrow 0$ to some
random variable $V$, where $\EE[V]$ may be infinite.
Then for all $\ell\in\R^d$,
\begin{equation}\label{cnn}
\liminf_{\ga\searrow 0}\frac{\beta_{V_\ga}(\ell)}{\sqrt{\ga}}\ge\sqrt{2d\,\EE[V]}\ 
\|\ell\|_2.
\end{equation}
\end{theorem}
%%%%%%%%%%%%%%%%%
Combining Theorems \ref{upp} and \ref{lowp}  with (\ref{jen}) we immediately obtain the following main result. 
%%%%%%%%%%%%%%%%%%%%%%%
\begin{theorem} 
\label{new} 
Assume that $V_\ga(0)/\ga$ converges in distribution to some
random variable $V$ and that $\EE[V_\ga(0)]/\ga\in]0,\infty[$ converges to $\EE[V]\in[0,\infty]$ as $\ga\searrow 0$. Then for all $\ell\in\R^d$,
\[
\lim_{\ga\searrow 0}\frac{\al_{V_\ga}(\ell)}{\sqrt{\ga}}=\lim_{\ga\searrow 0}\frac{\beta_{V_\ga}(\ell)}{\sqrt{\ga}}=
\sqrt{2d\, \EE[V]}\ \|\ell\|_2.
\]
\end{theorem}
%%%%%%%%%%%%%%%%%%%%
\begin{ex}\label{ex1}{\rm The simplest way to let $V_\ga(0)/\ga$ converge in distribution is to choose $V_\ga(0)=\ga V(0)$.
If we additionally assume $\EE[V(0)]<\infty$ then Theorem \ref{new} yields for all $\ell\in\R^d$,
\[
\lim_{\ga\searrow 0}\frac{\al_{\ga V}(\ell)}{\sqrt{\ga}}=\lim_{\ga\searrow 0}\frac{\beta_{\ga V}(\ell)}{\sqrt{\ga}}=
\sqrt{2d\, \EE[V(0)]}\ \|\ell\|_2.
\]
}
\end{ex}
\begin{ex}\label{ex2}{\rm To obtain Theorem \ref{cor} from Theorem \ref{new}
one needs to choose $V_\ga(0)=\ln(\ga V(0)+1)$. Then $V_\ga(0)/\ga$ converges 
a.s.\ to $V(0)$. Moreover, if $\EE[V(0)]<\infty$ then 
$\EE[V_\ga(0)]/\ga$ converges to $\EE[V(0)]$ by dominated convergence since $V_\ga(0)/\ga\le V(0)$. If $\EE[V(0)]=\infty$
then $\EE[V_\ga(0)]/\ga$ converges to $\EE[V(0)]$ as well by Fatou's lemma. Therefore,
Theorem \ref{new} together with (\ref{UV2}) and (\ref{sim2}) gives Theorem \ref{cor}.}
\end{ex}
The next two examples show that the conditions of Theorem \ref{new} are essential.
\begin{ex}\label{ex3}
{\rm Let $\PP[V_\ga(0)=0]=1-\ga$ and $\PP[V_\ga(0)=1]=\ga$. Then $\EE[V_\ga(0)]/\ga=1$  but $V_\ga(0)/\ga$ converges to 0 in probability as $\ga\searrow 0$. We shall show that in dimension one  $\al_{V_\ga}$ converges to zero faster than $\sqrt{\ga}$.

The ergodic theorem implies (see e.g.\ \cite[Proposition 10 (39)]{Ze98}) that $\al_{V_\ga}(1)=\EE[-\ln e(0,1,V_\ga)]$.
On the event $\{V_\ga(0)=1\}$ we have
$e(0,1,V_\ga)\ge e^{-1}P[S(1)=1]=(2e)^{-1}$. On the event $\{V_\ga(0)=0\}$ the quantity $e(0,1,V_\ga)$ is bounded below by the probability that the walk reaches 1 before it hits $-M$, where $M:=\inf\{m\ge 1\mid V_\ga(-m)=1\}$. Using that $M$ is geometrically distributed with parameter $\ga$ we obtain
\begin{eqnarray*}
\al_{V_\ga}(1)&\le& -\ln((2e)^{-1})\,\PP[V_\ga(0)=1]+\EE\left[-\ln\left(\frac{M}{M+1}\right)\won_{V_\ga(0)=0}\right]\\
& \le & 
(\ln 2e)\ga+\EE[1/M]\ =\ (\ln 2e)\ga-(\ga\ln\ga) /(1-\ga)\ \le \ -2\ga\ln\ga
\end{eqnarray*}
for small $\ga$.
}
\end{ex}
\begin{ex}\label{ex4}
{\rm 
Let $\PP[V_\ga(0)=\ga]=1-\ga^{1/3}$ and $\PP[V_\ga(0)=1/\ga]=\ga^{1/3}$. Then $V_\ga(0)/\ga$ converges to 1 in probability, whereas its expectation does not tend to 1 but to infinity as $\ga\searrow 0$. We shall show that in dimension one $\beta_{V_\ga}$ does not converge to zero as fast as $\sqrt{\ga}$. Indeed, for $d=1$ the quantity  $e(0,n,V_\ga)$ can be bounded above by the product of the i.i.d.\ random variables $e^{-V_\ga(i)}$, $i=0,\ldots,n-1$. Therefore,
\[ \beta_{V_\ga}(1)
\ge
-\ln \EE[e^{-V_\ga(0)}]\ =\ -\ln\left((1-\ga^{1/3})e^{-\ga}+\ga^{1/3}e^{-1/\ga}\right).
\]
For $\ga$ small enough this is greater than $-\ln((1-\ga^{1/3})+\ga^{1/3}/2)\ge \ga^{1/3}/2$.
}
\end{ex}
In the
next section we introduce our two main tools which are based on the strong Markov property and scaling of random walks.
These tools will be used for the proofs of both the upper bound Theorem \ref{upp}
in Section \ref{upper} and the lower bound Theorem \ref{lowp}
in Section \ref{lower}. 
In the appendix we comment on the proofs of Propositions \ref{zp} and \ref{db}.
%%%%%%%%%%%%%%%%%%%%%%%%%%%%%%%%%%%%%%%%%%%%%%%%
\section{Two main tools}\label{tools}
%%%%%%%%%%%%%%%%%%%%%%%%%%%%%%%%%%%%%%%%%%%%%%%%
For $\ga>0$ and $\ell\in\R^d\backslash\{0\}$ we define the stopping times $T_k=T_k(\ga,\ell)$, $k\in\N_0$, by
setting 
\begin{equation}\label{T}
T_0:=0\quad \mbox{ and }\quad T_{k+1}:=\inf\left\{n>T_k\mid S(n)\cdot \ell\ge S(T_k)\cdot \ell+\ga^{-1/2}\right\}.
\end{equation}
Note that these stopping times are increasingly ordered and $P$-a.s.\ finite. 
%%%%%%%%%%%%%%%%%%%%%
\begin{lemma}\label{s} Let $\ga>0$ and $\ell\in\R^d\backslash\{0\}$. Then the 
vectors $(S(n)-S(T_k))_{T_k< n\le T_{k+1}}$, $k\in\N_0$, with values in $\bigcup_{i\in\N}(\Z^d)^i$
are i.i.d.\ under $P$. Moreover, for all $0\le k\le K$, 
\begin{equation}\label{fra}
\ga^{-1/2}(K-k)\le \left(S(T_K)-S(T_k)\right)\cdot \ell\le\left(\ga^{-1/2}+\|\ell\|_\infty\right)(K-k).
\end{equation}
\end{lemma}
%%%%%%%%%%%%%%%%%%%%
\begin{proof}
The first statement follows from the strong Markov property, see also Figure \ref{de}.
\begin{figure}[t]
\hspace*{-40mm}
\psfrag{0}{0}
 \psfrag{x1}{$S(T_1)$}
 \psfrag{x2}{$S(T_2)$}
\psfrag{x3}{$S(T_3)$}
\psfrag{ga}{$\ga^{-1/2}$}
\epsfig{figure=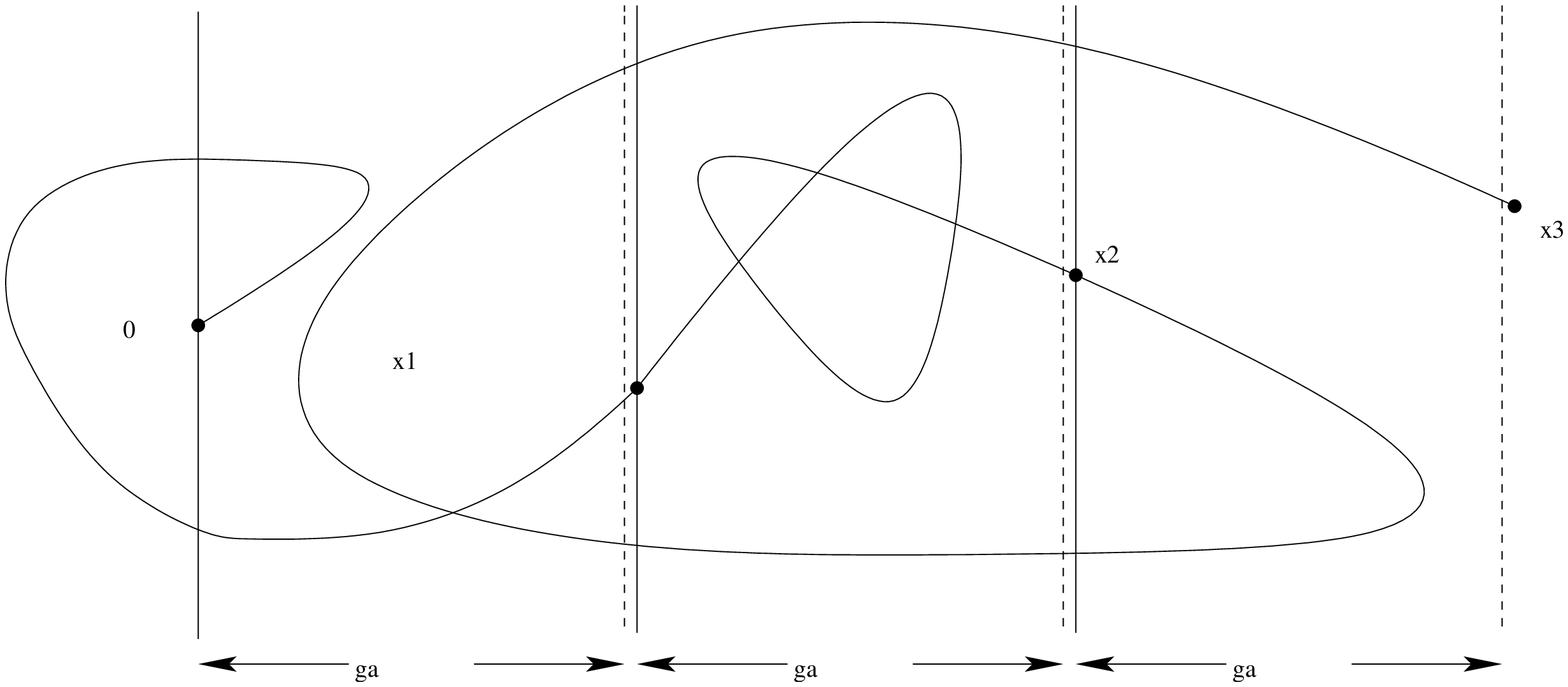,height=85pt}
\hspace*{5mm}
 \psfrag{AA}{$(S(n)-S(T_k))_{T_0\le n\le T_{1}}$}
 \psfrag{BB}{$(S(n)-S(T_k))_{T_1\le n\le T_{2}}$}
 \psfrag{CC}{$(S(n)-S(T_k))_{T_2\le n\le T_{3}}$}
\epsfig{figure=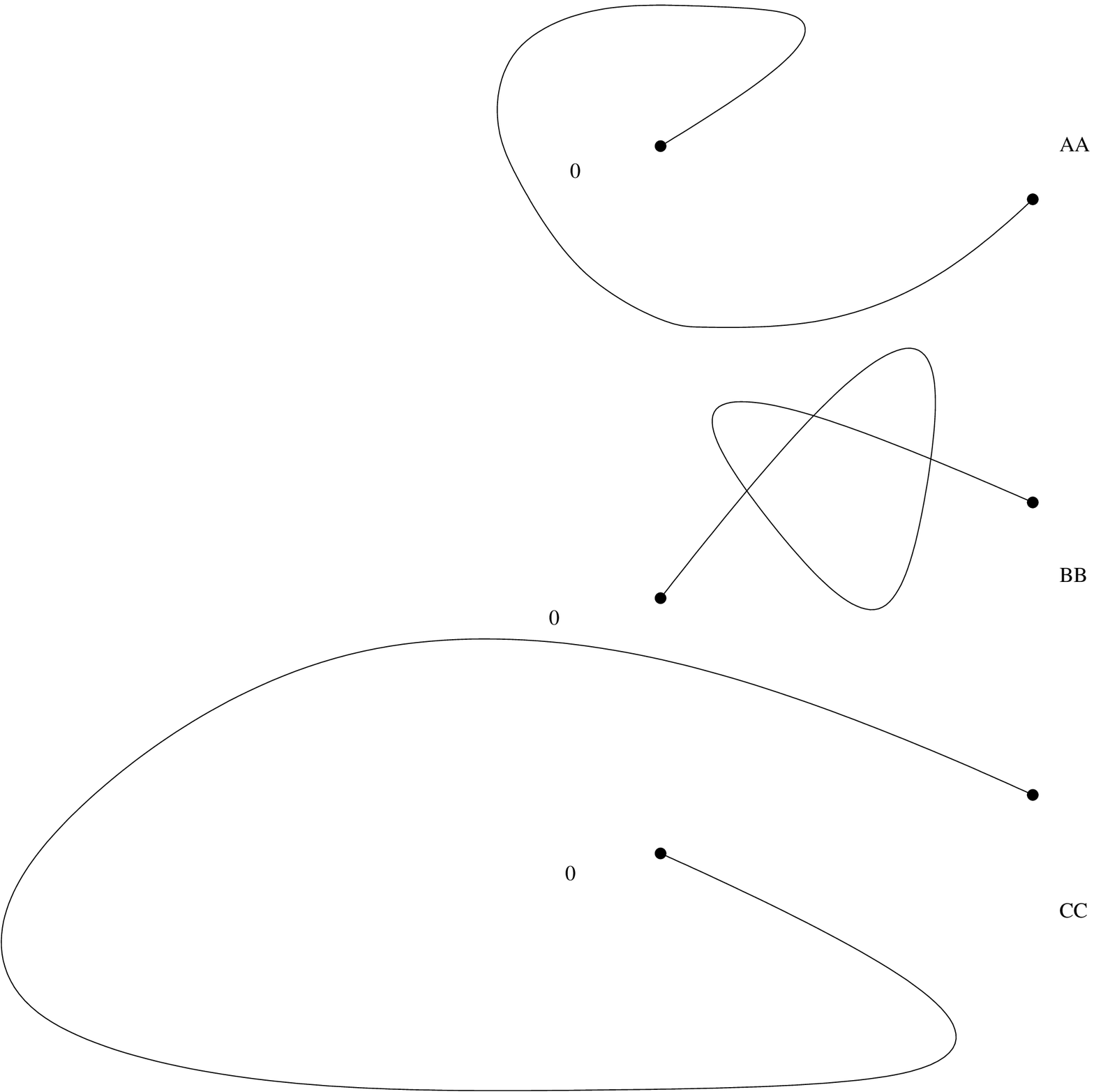,height=85pt}\vspace*{-0mm}
\caption{\footnotesize Decomposing the path $(S(n))_{n\ge 0}$ up to time $T_3$ into three pieces, which are i.i.d.\ after a shift. The boundary of the slabs of width $\ga^{-1/2}$ is solid on the left and dashed on the right. Due to lattice effects there is in general a small gap between neighboring slabs.}\label{de}
\end{figure}
The bounds in  (\ref{fra}) follow by induction over $K$, where we use for the upper bound that $(S(n))_{n\ge 0}$ is a nearest neighbor walk.
\end{proof}

The following lemma
explains the factor $\sqrt{2d}$ in Theorems \ref{cor} to  \ref{new}.
%%%%%%%%%%%%%%%%%%%%
\begin{lemma}\label{scale}
Let $\ell\in\R^d\backslash\{0\}$. Then $\ga  T_1(\ga, \ell)$ converges in distribution as $\ga\searrow 0$ to $dT/(\ell\cdot \ell)$,
where $T$ is  the hitting time of 1 for
a one-dimensional standard Brownian motion.
Therefore, for all  $c>0$,
\begin{equation}\label{wind}
\lim_{\ga\searrow 0}E\left[e^{-c \ga  T_1(\ga, \ell)}\right]\ =\ E\left[e^{-cd\, T/(\ell\cdot \ell)}\right]=e^{-\sqrt{2dc}/\|\ell\|_2}.
\end{equation}
\end{lemma}
%%%%%%%%%%%%%%%%%%%
\begin{proof}
Observe that 
\[X_n:=\sqrt{\frac{d}{\ell\cdot \ell}}\, S(n)\cdot \ell\]
defines a random walk $(X_n)_{n\ge 0}$ on $\R$, whose increments have mean 0 and variance 1. 
After rewriting  $T_1(\ga,\ell)$ as $
\inf\{n>0\mid X_n\ge \sqrt{d/(\ga \ell\cdot \ell)}\}$
the first statement follows from Donsker's invariance principle as explained 
e.g.\ in \cite[Example 7.6.6]{durr}. This immediately implies the first equality in (\ref{wind}).
The second one follows from the  
explicit expression for the Laplace transform of $T$, 
see e.g.\ \cite[(7.4.4)]{durr}.
\end{proof}
%%%%%%%%%%%%%%%%%%%%%%%%%%%%%%%%%%%%%%%%%%%%%%%%
\section{Proof of the upper bound}\label{upper}
%%%%%%%%%%%%%%%%%%%%%%%%%%%%%%%%%%%%%%%%%%%%%%%%%%%%%%%%%%
By Jensen's inequality, $V_\ga(0)$ is more variable than the constant $\EE[V_\ga(0)]$. 
Consequently, by Proposition \ref{zp},
\[\frac{\al_{V_\ga(0)}}{\sqrt{\ga}}\le \frac{\al_{\EE[V_\ga(0)]}}{\sqrt{\EE[V_\ga(0)]}}\ \sqrt{\frac{\EE[V_\ga(0)]}{\ga}}.\]
 For the proof of Theorem \ref{upp} it therefore suffices to show that 
\begin{equation}\label{up2}
\limsup_{\ga\searrow 0}\frac{\al_{\ga}(\ell)}{\sqrt{\ga}}\le
\sqrt{2d}\ \|\ell\|_2\quad\mbox{for all $\ell\in\R^d\backslash\{0\}$.}
\end{equation}
Observe that (\ref{up2}) is a statement about simple symmetric random walk only, without any reference to a random environment. 
One could prove (\ref{up2}) analytically by using
\cite[Theorem 21]{Ze98}, which states that
for all $\ga>0$ and all $\ell=(\ell_1,\ldots,\ell_d)\in\R^d\backslash\{0\}$,
\begin{equation}\label{form}
\al_\ga(\ell)=\sum_{i=1}^d \ell_i\, {\rm arsinh\, }(\ell_is),\quad\mbox{where $s>0$ solves}\quad
e^\ga d=\sum_{i=1}^d\sqrt{1+(\ell_is)^2}\, .
\end{equation}
However, since the proof of (\ref{form}) given in \cite{Ze98} is quite involved, we shall provide
an alternative proof of (\ref{up2}), which does not use (\ref{form}).
For this purpose, we  consider for constant potential $\ga>0$  the so-called point-to-hyperplane 
Lyapunov exponents
\begin{eqnarray}\label{bar}
\overline \al_\ga(\ell)&:=&\limsup_{k\to\infty}\frac{-1}{k}\ln 
E\left[\exp\left(-\ga\overline H(k\ell)\right)\right],\quad\mbox{where}\\
\overline H(\ell)&:=&\inf\{n\ge 0\mid S(n)\cdot \ell\ge \ell\cdot \ell\}\label{lara}
\end{eqnarray}
is the first time at which the random walk crosses the hyperplane that contains $\ell$ and is perpendicular to $\ell$.  Point-to-hyperplane exponents have been considered for constant potentials in the more general setting of random walks in random environments (RWRE) in \cite{Ze00}. For  random walks among random potentials they have been investigated in \cite{Fl07} and \cite{Zy09}. We shall show that
\begin{equation}\label{bart}
\lim_{\ga\searrow 0}\frac{\overline \al_\ga(\ell)}{\sqrt{\ga}}= \sqrt{2d}\ \|\ell\|_2\quad\mbox{for all $\ell\in\R^d\backslash\{0\}$.}
\end{equation}
By stopping the exponential martingale $\exp\left(\ga S(n)\cdot \ell-nf_\ell(\ga)\right)$ at time $\overline H(k\ell)$, where $f_\ell(\ga):= \ln E\left[e^{\ga S(1)\cdot \ell}\right],$ one could, in fact, show that $\overline \al_\ga(\ell)=f_\ell^{-1}(\ga) \ell\cdot \ell$ and deduce (\ref{bart}) from this. Instead we present in the following a different approach which uses the tools provided in Section \ref{tools}.
Fix $\ell\in\R^d\backslash\{0\}$ and set 
\begin{equation}\label{unizh}
m_k(\ga):=\Big\lfloor\frac{k\ell\cdot \ell}{\ga^{-1/2}+\|\ell\|_\infty}\Big\rfloor\quad\mbox{and}\quad 
M_k(\ga):=\big\lceil (k \ell\cdot \ell+\|\ell\|_\infty)\sqrt{\ga}\big\rceil
\end{equation}
for $k\in\N$ and $\ga>0$. Then for all $k$ and $\ga$,
\begin{eqnarray*}
S(\overline H(k\ell))\cdot \ell&\stackrel{(\ref{lara})}{\ge}& k\ell\cdot \ell\ \ge\ m_k(\ga)(\ga^{-1/2}+\|\ell\|_\infty)\stackrel{(\ref{fra})}{\ge}
S(T_{m_k(\ga)})\cdot \ell\quad\mbox{and}\\ \nonumber
S(\overline H(k\ell))\cdot \ell&=&
S(\overline H(k\ell)-1)\cdot \ell+(S(\overline H(k\ell))-S(\overline H(k\ell)-1))\cdot \ell\\
&\stackrel{(\ref{lara})}{\le}& k\ell\cdot \ell+\|\ell\|_\infty\ \le\  M_k(\ga)\ga^{-1/2}\ \stackrel{(\ref{fra})}{\le}\ S(T_{M_k(\ga)})\cdot \ell.\nonumber
\end{eqnarray*}
Hence,
\begin{equation}\label{epfl}
T_{m_k(\ga)}  \le\overline H(k\ell)\le T_{M_k(\ga)}.
\end{equation}
For any $m$ one can represent $T_m$ as telescopic sum $\sum_{i=1}^m T_i-T_{i-1}$ of random variables which measure the length of the sequence $(S(n)-S(T_{i-1}))_{T_{i-1}<n\le T_i}$ and are therefore i.i.d.\
due to Lemma \ref{s}.
We obtain from (\ref{epfl})
\[\left(E\left[e^{-\ga T_1}\right]\right)^{m_k(\ga)}\ge
E\left[\exp\left(-\ga\overline H(k\ell)\right)\right]\ge \left(E\left[e^{-\ga T_1}\right]\right)^{M_k(\ga)}.\]
Substituting this into definition (\ref{bar}) yields
\[\frac{-\ell\cdot \ell}{\ga^{-1/2}+\|\ell\|_\infty}\ln E[e^{-\ga T_1}]
\le \overline \al_\ga(\ell)\le -\sqrt{\ga} \ell\cdot \ell\ln E[e^{-\ga T_1}].\] 
The statement 
 (\ref{bart}) now follows from 
 Lemma \ref{scale}.

In order to derive from this our goal (\ref{up2}) we need to relate $\overline \al_\ga$ and $\al_\ga$.
Applying \cite[Lemma 2]{Ze00} to the simple symmetric random walk yields
\begin{equation}\label{inf}
\overline\al_\ga(\ell)=\inf\{\al_\ga(x)\mid x\in\R^d, x\cdot \ell\ge \ell\cdot \ell\}.
\end{equation}
(Here our $\overline \al_\la(\ell)$ and $\overline H(k\ell)$ correspond to $\ga_\la(\ell/(\ell\cdot\ell))$
and $T_k(\ell/(\ell\cdot\ell))$, respectively, in the notation used in \cite{Ze00}. See also \cite[Corollary C]{Fl07}, where $\overline \al_\ga$ is expressed in terms of the dual norm of 
 $\al_\ga$ and \cite[Proposition 2.2]{Zy09}.)
Since $\al_\ga$ is a norm, $\al_\ga(x)\to\infty$ as $\|x\|_2\to\infty$. Therefore, the infimum in (\ref{inf}) is attained, i.e.\ for all $\ell\ne 0$ and $\ga>0$
there is some $x(\ga,\ell)\ne 0$ such that
\begin{equation}\label{nyu}
x(\ga,\ell)\cdot \ell=\ell\cdot \ell\qquad\mbox{and}\qquad\overline \al_\ga(\ell)=\al_\ga(x(\ga,\ell)).
\end{equation}
For example, if we denote by $e_1,\ldots,e_d$ the canonical basic vectors of $\Z^d$, then $x(\ga,e_1)$ can be chosen so that
\begin{equation}\label{eg}
x(\ga,e_1)=e_1,\quad \mbox{i.e.}\quad \overline\al_\ga(e_1)=\al_\ga(e_1).
\end{equation}
Indeed, since $\al_\ga$ is invariant under the reflection $x=(x_1,\ldots,x_d)\mapsto 2x_1e_1-x$ 
we have
\[\al_\ga(e_1)= \frac{\al_\ga(2e_1)}{2}\le \frac{\al_\ga(2e_1-x(\ga,e_1))+\al_\ga(x(\ga,e_1))}{2}\stackrel{(\ref{nyu})}{=}
\alpha_\gamma (x(\ga,e_1))\stackrel{(\ref{nyu})}{=}\overline \al_\ga(e_1).
\]
Using the norm and invariance properties of $\al_\ga$ again we obtain from this example
\begin{eqnarray}\nonumber
\limsup_{\ga\searrow 0}\sup_{\|x\|_2=1}\frac{\al_\ga(x)}{\sqrt{\ga}}
&\le&  
\limsup_{\ga\searrow 0}\sup_{\|x\|_2=1}\frac{\sum_{i=1}^d|x_i|\al_\ga(e_i)}{\sqrt{\ga}}\\
& =& \label{oma}
\limsup_{\ga\searrow 0}\frac{\al_\ga(e_1)}{\sqrt{\ga}}\sup_{\|x\|_2=1}\|x\|_1\\
&\stackrel{(\ref{eg})}{=}& 
\limsup_{\ga\searrow 0}\frac{\overline \al_\ga(e_1)}{\sqrt{\ga}}\sup_{\|x\|_2=1}\|x\|_1
 \ \stackrel{(\ref{bart})}{=}\ \nonumber
\sqrt{2d}\, \sup_{\|x\|_2=1}\|x\|_1\ <\ \infty.
\end{eqnarray}
Except for the constant $\sup_{\|x\|_2=1}\|x\|_1$ this already gives the correct behavior claimed in (\ref{up2}). To get the
right constant we 
next  show that 
\begin{equation}\label{kmz}
\lim_{\ga\searrow 0}x(\ga,\ell)=\ell\quad\mbox{for all $\ell\in\R^d\backslash\{0\}$}.
\end{equation}
Assume that there are $\ell\in\R^d\backslash\{0\}$, $\eps>0$, and a sequence $(\ga_n)_{n\ge 0}$ tending to 0 such that $\|x_n-\ell\|_2\ge\eps$
for all $n$, where $x_n:=x(\ga_n,\ell)$. Due to compactness  we may assume  without loss of generality that $x_n/\|x_n\|_2$ converges to some $z$ 
as $n\to\infty$. Then
\begin{eqnarray*}
\sqrt{2d} &\stackrel{(\ref{bart})}{=}& \lim_{n\to\infty}\frac{\overline\al_{\ga_n}(z)}{\sqrt{\ga_n}}
\ \stackrel{(\ref{inf})}{\le}\ \liminf_{n\to\infty}\frac{\al_{\ga_n}(z)}{\sqrt{\ga_n}}\\  
&\le&   \liminf_{n\to\infty}\frac{\al_{\ga_n}(x_n)}{\sqrt{\ga_n}\,\|x_n\|_2}+  \frac{1}{\sqrt{\ga_n}}\al_{\ga_n}\left(z-\frac{x_n}{\|x_n\|_2}\right)\\
&\stackrel{(\ref{nyu})}{\le}& \liminf_{n\to\infty}
\frac{\overline\al_{\ga_n}(\ell)}{\sqrt{\ga_n}\,\|x_n\|_2}+
\left\|z-\frac{x_n}{\|x_n\|_2}\right\|_2\sup_{\|y\|_2=1}\frac{\al_{\ga_n}(y)}{\sqrt{\ga_n}}\\
&\stackrel{(\ref{bart}),(\ref{oma})}{=}&\sqrt{2d}\,\|\ell\|_2\liminf_{n\to\infty}
\frac{1}{\|x_n\|_2}+0\ \le\ \sqrt{2d}\,\frac{\|\ell\|_2}{\sqrt{\|\ell\|_2^2+\eps^2}}<\sqrt{2d},
\end{eqnarray*}
where we used in the second to last step that $\|x_n-\ell\|_2\ge \eps$, $x_n\cdot \ell=\ell\cdot \ell$ and 
the Pythagorean theorem. This gives the desired contradiction and proves (\ref{kmz}). Now,
\begin{eqnarray*}
\al_\ga(\ell)&=& \al_\ga(x(\ga,\ell)) +\al_\ga(\ell)-\al_\ga(x(\ga,\ell)) \ \stackrel{(\ref{nyu})}{\le}\ 
\overline\al_\ga(\ell)+\al_\ga(\ell-x(\ga,\ell))\\
&\le&\overline\al_\ga(\ell)+\|\ell-x(\ga,\ell)\|_2 \sup_{\|y\|_2=1}\al_\ga(y).
\end{eqnarray*}
Therefore,
\[\limsup_{\ga\searrow 0}\frac{\al_\ga(\ell)}{\sqrt{\ga}}\le \limsup_{\ga\searrow 0}\frac{\overline\al_\ga(\ell)}{\sqrt{\ga}}+\|\ell-x(\ga,\ell)\|_2 \sup_{\|y\|_2=1}\frac{\al_\ga(y)}{\sqrt{\ga}}\ =\ \sqrt{2d}\, \|\ell\|_2
\]
due to (\ref{bart}), (\ref{oma}) and (\ref{kmz}). This completes the proof of (\ref{up2}).
%%%%%%%%%%%%%%%%%%%%%%%%%%%%%%%%%%%%%%%%%%%%%%%%
\section{Proof of the lower bound}\label{lower}
%%%%%%%%%%%%%%%%%%%%%%%%%%%%%%%%%%%%%%%%%%%%%%%%
We first argue why it is enough to prove (\ref{cnn}) for $\ell\in\Z^d$. 
Homogeneity of the norms $\beta_{V_\ga}$ implies that (\ref{cnn}) then also holds for $\ell\in\Q^d$. Now let $\ell\in\R^d\backslash\{0\}$ be arbitrary and $\eps>0$. 
There are isometries $f_0,\ldots,f_{d}$ of $\Z^d$, one of which being the identity, which preserve the origin and for which the convex hull of $\{f_i(\ell)\mid i=0,\ldots,d\}$ has interior points. Since this hull is a polytope there  are
coefficients  $a_0,\ldots,a_d\in[0,1]$ summing up to 1 which define an  $x:=\sum_{i=0}^{d}a_if_i(\ell)$ such that $x \in\Q^d$ and $\|x-\ell\|_2<\eps$.
Therefore,
\begin{eqnarray*}
\sqrt{2d\,\EE[V]}(\|\ell\|_2-\eps)&\le&\sqrt{2d\,\EE[V]}\|x\|_2 
\ \le\ \liminf_{\ga\searrow 0}\frac{\beta_{V_\ga}(x)}{\sqrt{\ga}}\\
&\le& \liminf_{\ga\searrow 0}\frac{\sum_{i=0}^da_i\beta_{V_\ga}(f_i(\ell))}{\sqrt{\ga}}
\ =\ \liminf_{\ga\searrow 0}\frac{\beta_{V_\ga}(\ell)}{\sqrt{\ga}}.
\end{eqnarray*}
Letting $\eps\searrow 0$ yields the claim (\ref{cnn}).

It remains to prove (\ref{cnn}) for $\ell\in\Z^d$.
Without loss of generality we may assume $\EE[V]>0$.
For $0\le k<n$ and $x\in\Z^d$ we denote by 
\[\ell_k^n(x):=\#\{m\in\N_0\mid k\le m<n,\ S(m)=x\}\]
the local time of the random walk in $x$ between times $k$ and $n$. We use these local times to rewrite the
definition of the annealed exponent $\beta_{V_\ga}$, cf.\ \cite[p.\ 597, 598]{Fl07}. We have
\[\EE\left[E\left[\exp\left(-\sum_{m=0}^{H(k\ell)-1}V_\ga (S(m))\right)\right]\right]
=E\left[\EE\left[\exp\left(-\sum_{x\in\Z^d}V_\ga(x) \ell_0^{H(k\ell)}(x)\right)\right]\right].\]
Using that $(V(x))_{x\in\Z^d}$ is i.i.d.\ under $\PP$ we obtain that the last expression is equal to
\[E\left[\,\prod_{x\in\Z^d}\EE\left[\exp\left(-V_\ga(x)\ell_0^{H(k\ell)}(x)\right)\right]\right]
=E\left[\exp\left(-\sum_{x\in\Z^d}\La_\ga\left(\ga\ell_0^{H(k\ell)}(x)\right)\right)\right],\]
where
\[\La_\ga(\la):=-\ln\EE\left[e^{-\la V_\ga(0)/\ga}\right]\qquad (0\le\la<\infty).\]
Therefore, $\beta_{V_\ga}(\ell)$, defined in (\ref{be}), can be expressed as
\begin{equation}\label{can}
\beta_{V_\ga}(\ell)=\lim_{k\to\infty}\frac{-1}{k}\ln 
E\left[\exp\left(-\sum_{x\in\Z^d}\La_\ga\left(\ga\ell_0^{H(k\ell)}(x)\right)\right)\right].
\end{equation}
In view of this representation of $\beta_{V_\ga}$ the claim (\ref{cnn})
is, given the functions $\La_\ga$, 
a statement only about the simple symmetric random walk, without any reference to a random environment. For the proof of (\ref{cnn}) define the function
$\La(\la):=
-\ln\EE\left[e^{-\la V}\right]$ and fix $0<t<\EE[V]$. Since $\La(0)=0$ and $\La'(0)=\EE[V]$ there is
$\la_0>0$ such that $\La(\la_0)/\la_0>t$, see Figure \ref{ff}.
\begin{figure}[t]
\begin{center}
 \psfrag{la}{$\la$}
 \psfrag{La}{$\La(\la)$}
 \psfrag{la0}{$\la_0$}
\epsfig{figure=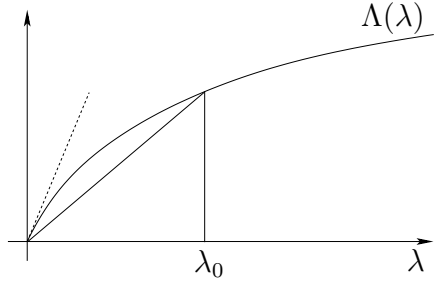,height=100pt}\vspace*{-0mm}
\end{center}
\caption{\footnotesize  The  dashed tangent line has slope $\La'(0)=\EE[V(0)].$ For each $t<\EE[V(0)]$
we can choose $\la_0$ such that the slope of the secant is at least $t$.}\label{ff}
\end{figure}
By weak convergence, $\La_\ga(\la_0)$ converges to $\La(\la_0)$ as $\ga\searrow 0$.
Therefore, there is $\ga_0>0$ such that  
\begin{equation}\label{ias}
\frac{\La_\ga(\la_0)}{\la_0}> t\qquad \mbox{for all $0<\ga\le\ga_0$}.
\end{equation}
However, each $\La_\ga$ is concave  because of $\La_\ga''(\la)=-{\rm Var}(V_\ga(0)/\ga)\le 0,$
where the variance is taken with respect to the normalized measure 
$e^{-\la V_\ga(0)/\ga}\, d\PP$.
Therefore, since $\La_\ga(0)=0$, the function $\La_\ga(\la)/\la$ decreases in $\la$. Together with (\ref{ias}) this implies that
\begin{equation}\label{iad}
\frac{\La_\ga(\la)}{\la}>t\qquad\mbox{ for all $0<\la\le \la_0$ and $0<\ga\le\ga_0$.}
\end{equation}
Now let 
\begin{equation}\label{sued}
0<\ga\le\min\{1,\ga_0,(\la_0/3)^8\}
\end{equation}
and define $m_k=m_k(\ga)$ for $k\in\N_0$  as in (\ref{unizh}). Then by (\ref{epfl}), 
$H(k\ell)\ge\overline H(k\ell)\ge T_{m_k}(\ga,\ell)$.
Therefore, by (\ref{can}),
\begin{equation}\label{sun}
\beta_{V_\ga}(\ell)\ge \limsup_{k\to\infty}\frac{-1}{k}\ln 
E\left[\exp\left(-\sum_{x\in\Z^d}\La_\ga\left(\ga\ell_0^{T_{m_k}}(x)\right)\right)\right].
\end{equation}
For $x\in\Z^d$ set $i_x:=\min\{i\ge 0\mid x\cdot \ell<S(T_i)\cdot \ell\}.$
We expand $\ell_0^{T_{m_k}}(x)$ as a telescopic sum and then omit some  of the summands and truncate the 
remaining ones to obtain for all $k\ge 0$ and $x\in\Z^d$,
\begin{equation}\label{most}
\ga\,\ell_0^{T_{m_k}}(x)\ =\ \ga\sum_{i=1}^{m_k}\ell_{T_{i-1}}^{T_i}(x)\ \ge \ 
\ga\sum_{i=1\vee i_x}^{m_k\wedge \lceil i_x+\ga^{-1/8}\rceil}\left(\ell_{T_{i-1}}^{T_i}(x)\wedge \ga^{-3/4}\right) .
\end{equation}
The right most side of (\ref{most}) is less than or equal to
\[\ga (\ga^{-1/8}+2)\ga^{-3/4}=\ga^{1/8}+2\ga^{1/4}\le3\ga^{1/8}\stackrel{(\ref{sued})}{\le}\la_0.\]
Hence (\ref{iad}) can be applied, which shows that  for all $k\ge 0$ and $x\in\Z^d$,
\[\La_\ga\left(\ga\,\ell_0^{T_{m_k}}(x)\right)\ge t
\ga\sum_{i=1\vee i_x}^{m_k\wedge \lceil i_x+\ga^{-1/8}\rceil}\left(\ell_{T_{i-1}}^{T_i}(x)\wedge \ga^{-3/4}\right) .\]
Summing over $x\in\Z^d$, changing the order of summation and omitting some more summands gives for all $k\ge 0$,
\begin{eqnarray}\label{Y}
\sum_{x\in\Z^d}\La_\ga\left(\ga\,\ell_0^{T_{m_k}}(x)\right)&\ge& t
\ga\sum_{i=1}^{m_k}Y_i\, ,\qquad\mbox{where}\\
\nonumber
Y_i&:=&\sum_{x\in \mathcal S_i}
\left(\ell_{T_{i-1}}^{T_i}(x)\wedge \ga^{-3/4}\right)\qquad\mbox{and}\\
\nonumber
\mathcal S_i&:=&\{x\in\Z^d\mid -\ga^{-5/8}\le (x-S(T_{i-1}))\cdot \ell<\ga^{-1/2}\}.
\end{eqnarray}
\begin{figure}[t]
\psfrag{0}{0}
 \psfrag{a}{$\ell_{T_{i-1}}^{T_i}(x)$}
 \psfrag{d}{$\ga^{-3/4}$}
\psfrag{c}{$S(T_{i-1})$}
\psfrag{b}{$S(T_i)$}
\psfrag{e}{$\ga^{-5/8}$}
\psfrag{f}{$\ga^{-1/2}$}
\psfrag{x}{$x$}
\epsfig{figure=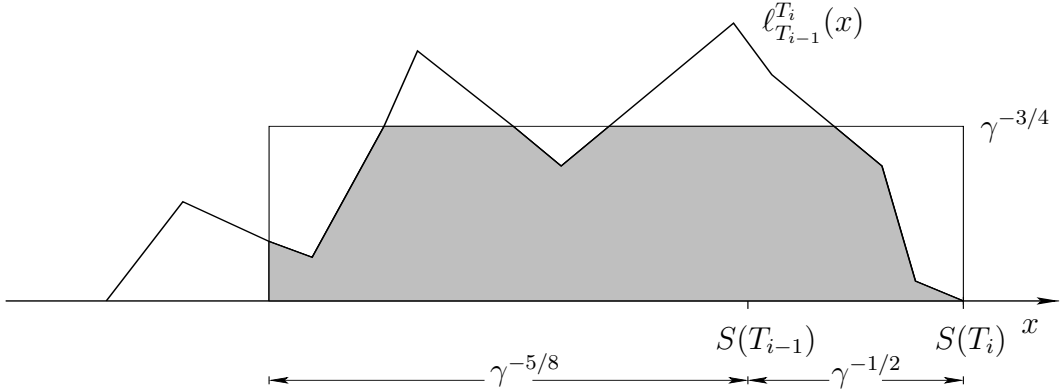,height=140pt, width=400pt}
\caption{\footnotesize The function  $\ell_{T_{i-1}}^{T_i}(x)$ is sketched as a zigzag graph. 
The box is the indicator function, multiplied by $\ga^{-3/4}$, of the set $\mathcal S_i$.
The size of the shaded area is $Y_i$. Typically, for small $\ga>0$, unlike in the figure,
the zigzag graph fits into the box, in which case $Y_i=T_i-T_{i-1}$.}\label{box}
\end{figure}
Indeed, if $1\le i\le {m_k}$ and $x\in\mathcal S_i$ then on the one hand
\[x\cdot \ell<S(T_{i-1})\cdot \ell +\ga^{-1/2}\stackrel{(\ref{fra})}{\le}S(T_{i-1})\cdot \ell +\left(S(T_i)-S(T_{i-1})\right)\cdot \ell
=S(T_i)\cdot \ell\]
and therefore $i_x\le i$ and on the other hand
\begin{eqnarray*}
x\cdot \ell&\ge& S(T_{i-1})\cdot \ell-\ga^{-5/8}\\
&=&S(T_{i-1-\lceil \ga^{-1/8}\rceil})\cdot \ell+\left(S(T_{i-1})- S(T_{i-1-\lceil \ga^{-1/8}\rceil})\right)\cdot \ell- \ga^{-5/8}\\
&\stackrel{(\ref{fra})}{\ge}&S(T_{i-1-\lceil \ga^{-1/8}\rceil})\cdot \ell
\end{eqnarray*}
and therefore $i_x\ge i-\lceil \ga^{-1/8}\rceil$, i.e.\ $i\le \lceil i_x+\ga^{-1/8}\rceil$, which concludes the proof of (\ref{Y}). For an illustration in the one-dimensional case see Figure \ref{box}.

Now the key observation is  that there is a function $f:\bigcup_{j\in\N}(\Z^d)^j\to\R$ such that 
\[Y_i=f\left((S(n)-S(T_{i-1}))_{T_{i-1}<n
\le T_i}\right)
\]
 for all $i$.  
Therefore, Lemma \ref{s} implies that the sequence $(Y_i)_{i\in\N}$ is i.i.d.\ under $P$.
Consequently, it follows from (\ref{sun}),  (\ref{Y}) and (\ref{unizh}) that
\begin{equation}\label{sinn}
\beta_{V_\ga}(\ell)\ge \frac{-\ell\cdot \ell}{\ga^{-1/2}+\|\ell\|_\infty}\ln E\left[\exp\left(-t\ga Y_1\right)\right].
\end{equation}
Observe that $Y_1=T_1$ on \[\{S(n)\cdot \ell> -\ga^{-5/8} \text{ for all }n<T_1\}\cap \{\ell_0^{T_1}(x)\le \ga^{-3/4}\text{ for all } x\in\Z^d\}.\] Therefore, we obtain from (\ref{sinn}) by a union bound that
\begin{eqnarray}\nonumber
\frac{\beta_{V_\ga}(\ell)}{\sqrt{\ga}}&\ge& \frac{-\ell\cdot \ell}{1+\sqrt{\ga}\|\ell\|_\infty}\ln \bigg(
E\left[\exp\left(-t\ga\, T_1\right)\right]\\ \label{end2}
&& +\ P\left[\exists n<T_1\ S(n)\cdot \ell\le -\ga^{-5/8}\right]+P\left[\sup_{x\in\Z^d}\ell_0^{T_1}(x)>\ga^{-3/4}\right]\bigg).
\end{eqnarray}
Now we let $\ga\searrow 0$.
It suffices to show that then both terms in line (\ref{end2}) vanish.
Indeed, then the claim of this section, (\ref{cnn}), follows from applying Lemma \ref{scale}  and 
letting $t\nearrow \EE[V]$.

A variation of the gambler's ruin problem shows that the first term in (\ref{end2}) tends to zero.
The second term is less than or equal to
\begin{equation}\label{swr}
P\left[\ga T_1\ge \ga^{-1/4}\right]+P\left[\sup_{x\in\Z^d}\ell_0^{\ga^{-5/4}}(x)>\ga^{-3/4}\right].
\end{equation}
The first statement of Lemma \ref{scale} implies that the first term in (\ref{swr}) vanishes as $\ga\searrow 0$.
Concerning the second term in  (\ref{swr}), 
the literature contains precise and deep statements about the asymptotics of the maximal local time $\sup_{x}\ell_0^{n}(x)$ up to time $n$ as $n$ goes to infinity from which one could see that this term  also tends to 0.
However, this can also be derived by more elementary means as follows.
Note that  the second term in (\ref{swr}) is equal to
\[
P\left[\exists\, i<\ga^{-5/4}:\ \ell_0^{\ga^{-5/4}}(S(i))>\ga^{-3/4}\right]
\le P\left[\exists\, i<\ga^{-5/4}:\ \ell_i^{i+\ga^{-5/4}}(S(i))>\ga^{-3/4}\right].
\]
Using subadditivity and the Markov property, the last expression can be estimated from above by $\lceil\ga^{-5/4}\rceil P\left[\ell_0^{\ga^{-5/4}}(0)>\ga^{-3/4}\right]$. If we denote by 
$r_i\ (i\ge 1)$ the time between the $i$-th and the $(i+1)$-st visit to 0 this bound can be rewritten as 
\begin{equation}\label{snow}
\lceil\ga^{-5/4}\rceil P\bigg[\sum_{i=1}^{\lfloor\ga^{-3/4}\rfloor} r_i<\ga^{-5/4}\bigg]\ \le\ 
\lceil\ga^{-5/4}\rceil P\left[\bigcap_{i=1}^{\lfloor\ga^{-3/4}\rfloor}  \{r_i<\ga^{-5/4}\}\right].
\end{equation}
By the strong Markov property $(r_i)_{i\ge 1}$ is i.i.d.. Hence the right hand side of (\ref{snow}) equals
\begin{equation}\label{is}
\lceil\ga^{-5/4}\rceil \left(1-P\left[ r_1\ge \ga^{-5/4}\right]\right)^{\lfloor\ga^{-3/4}\rfloor}.
\end{equation}
One way for the random walk not to return to its starting point before time $\ga^{-5/4}$  is to reach the hyperplane at distance $\lceil\ga^{-5/8}\rceil$ in direction $e_1$ before returning to the hyperplane $Z:=\{x\mid x\cdot e_1=0\}$  and then to take more than $\ga^{-5/4}$ steps before returning to $Z$. Therefore, using the notation introduced in (\ref{T}),
\begin{equation}\label{per}
P\left[ r_1\ge \ga^{-5/4}\right]\ge P\left[\bigcap_{n=1}^{T_1(\ga^{5/4},e_1)} \{S(n)\cdot e_1>0\}\right] P\left[T_1(\ga^{5/4},e_1)>\ga^{-5/4}\right].
\end{equation}
By the gambler's ruin problem, the first term on the right hand side of (\ref{per}) equals
$1/(2d\lceil\ga^{-5/8}\rceil)$, whereas the second term tends to a constant $c>0$ due to Lemma \ref{scale}.
Consequently, the expression in (\ref{is}) is bounded from above for small $\ga>0$ by 
\[\lceil\ga^{-5/4}\rceil \left(1-c\ga^{5/8}/(3d)\right)^{\ga^{-3/4}}
\ \le\ \lceil\ga^{-5/4}\rceil e^{-(c/(3d))\ga^{5/8}\ga^{-3/4}},\]
which decays to 0 as $\ga\searrow 0$ since $5/8-3/4=-1/8<0$.
%%%%%%%%%%%%%%%%%%%%%%%%%%%%%%%%%%%%%%%%%%%%
\section{Appendix}
%%%%%%%%%%%%%%%%%%%%%%%%%%%%%%%%%%%%%%%%%%%%
\begin{proof}[Proof of Proposition \ref{zp}.]
All the statements of Proposition \ref{zp} are contained in \cite[Proposition 4]{Ze98}
except for (\ref{do2}). For the proof of (\ref{do2}) we recall from \cite[(10)]{Ze98} that
\begin{equation}\label{ppp}
g(0,k\ell,V)=e(0,k\ell,V)\,g(k\ell,k\ell,V),
\end{equation}
where $g(x,y,V)$ is defined like in (\ref{gr}) except that the random walk does not start at 0 but at $x$.
From this it follows, see \cite[(19)]{Ze98}, that 
\begin{equation}\label{emi}
|\ln g(0,k\ell,V)-\ln e(0,k\ell,V)|\le V(k\ell)+\left(V(k\ell)+\ln g(k\ell,k\ell,V)\right).
\end{equation}
Divided by $k$ this converges a.s.\ to 0 as $k\to\infty$ due to $\EE[V(0)]<\infty$ and   \cite[Lemma 5]{Ze98}.
This proves the first equality in (\ref{do2}). For the second identity  in (\ref{do2}) observe that by  (\ref{emi}) and translation invariance,
 \[\left|\EE\left[\ln g(0,k\ell,V)-\ln e(0,k\ell,V)\right]\right|\le
\EE[V(0)]+\EE\left[V(0)+\ln g(0,0,V)\right],\]
which is finite due to \cite[Lemma 5]{Ze98} and thus converges to 0 after division by $k\to\infty$.
\end{proof}
\begin{proof}[Proof of Proposition \ref{db}.]
All the statements of Proposition \ref{db} are contained in \cite[Theorem A (b)]{Fl07}
except for the last one and the second equality in (\ref{be}).
However, the last statement of Proposition \ref{db} is obvious. The second identity in (\ref{be})
is stated in \cite[(1.5),(1.6)]{Zy09} without proof. Since we could not find any proof of this identity in the 
literature we provide one here.

For the inequality $\le$ we observe that choosing $m=0$ in definition (\ref{gr}) gives $g(k\ell,k\ell,V)\ge e^{-V(k\ell)}$. Since $V(k\ell)$ and $e(0,k\ell,V)$ are independent we obtain from (\ref{ppp}) and translation invariance of $\PP$ that
$\EE[g(0,k\ell,V)]\ge \EE[e(0,k\ell,V)]\,\EE\left[e^{-V(0)}\right],$ which gives the desired inequality. 

For the opposite inequality we quote from \cite[(18)]{Ze98} that
by the strong Markov property $e^{V(k\ell)}g(k\ell,k\ell,V)$ can be represented as a geometric series such that
\begin{equation}\label{geo}
g(k\ell,k\ell,V)=e^{-V(k\ell)}\Bigg(1-E_{k\ell}\Bigg[\exp\Bigg(-\sum_{m=0}
^{H_2(k\ell)-1}V(S(m))\Bigg)\won_{H_2(k\ell)<\infty}\Bigg]\Bigg)^{-1},
\end{equation}
where $H_2(k\ell)$ denotes the time of the second visit of $k\ell$ and the subscript $k\ell$ indicates the starting point of  the random walk. Now fix $\eps>0$ such that $\PP[V(0)\ge \eps]\ge \eps$ and define the random set $A:=\{x\in\Z^d\mid V(x)\ge \eps\}$ and its entrance time $H(A)$.
Then the expectation in (\ref{geo}) can be bounded above by
\[
P_{k\ell}[H_2(k\ell)\le H(A)]+e^{-\eps}P_{k\ell}[H_2(k\ell)>H(A)]=
1-(1-e^{-\eps})\,P_{k\ell}[H_2(k\ell)>H(A)].
\]
Substituting this
into (\ref{geo}) gives
$g(k\ell,k\ell,V)\le c\,P_{k\ell}[H_2(k\ell)>H(A)]^{-1}$, where $c:=(1-e^{-\eps})^{-1}<\infty$. By (\ref{ppp}),
\begin{eqnarray}\nonumber
\EE[g(0,k\ell,V)]&\le& c\,\EE\left[e(0,k\ell,V)\,P_{k\ell}[H_2(k\ell)>H(A)]^{-1}\right]\\
&\le &c\,e^{k^{2/3}}\EE\left[e(0,k\ell,V)\right]+c\,\EE\left[P_{0}[H_2(0)>H(A)]^{-1}\,\won_{A_k}\right],
\label{sec}
\end{eqnarray}
where $A_k$ denotes the event that $P_{0}[H_2(0)>H(A)]\le e^{-k^{2/3}}$. As required, the first term in (\ref{sec}) has the same exponential decay rate for $k\to\infty$ as $\EE\left[e(0,k\ell,V)\right]$. Therefore, it suffices to show that the second term in (\ref{sec}) decays even superexponentially fast. For this purpose we denote by $L$
the $|\cdot|_1$-distance between 0 and $A$. For $d=1$, by the gambler's ruin problem,
$P_{0}[H_2(0)>H(A)]$ can be estimated from below by $1/(2(L\vee 1))$. For $d\ge 2$ one can bound this term very roughly from below by $(2d)^{-L}$ by choosing a path of minimal length connecting 0 to $A$.
Therefore,
\begin{equation}\label{am}
\EE\left[P_{0}[H_2(0)>H(A)]^{-1}\,\won_{A_k}\right]\le 
\left\{\begin{array}{ll}
{\displaystyle 2\EE\left[(L\vee 1)\,\won_{2(L\vee 1)\ge \exp(k^{2/3})}\right]}&\mbox{if $d=1,$}\\
{\displaystyle \EE\left[(2d)^{L}\, \won_{(2d)^L\ge \exp(k^{2/3})}\right]}&\mbox{if $d\ge 2$.}
\end{array}\right.
\end{equation}
Since the volume of the discrete $|\cdot|_1$-ball of radius $n$ centered at 0 is less than $c_dn^d$ for some suitable finite constant
$c_d$ depending on the dimension we get $\PP[L\ge n]\le(1-\eps)^{c_dn^d}$. 
Therefore, for $d\ge 2$ the expression on the right hand side of (\ref{am}) is less than 
$\sum_n(2d)^n(1-\eps)^{c_dn^d}$, where $n\ge n_k:=k^{2/3}/\ln(2d)$. For $k$ large enough this is less than
$\sum_n(1-\eps/2)^{c_dn^d}$.  Since $n^d\ge n_k^d+n-n_k$ for $n\ge n_k$ this can be bounded from above by
 $C_d(1-\eps/2)^{c_d n_k^d}$, where $C_d:=\sum_{n\ge 0}(1-\eps/2)^{c_d n}$, and thus decays 
superexponentially fast to zero.
The corresponding proof for $d=1$ is even simpler.
\end{proof}
\vspace*{3mm}

{\bf Acknowledgment:} Part of this work was done during the authors' stay at the Institut Henri Poincare - Centre Emile Borel. The authors thank this institution for
hospitality. E.\ Kosygina's work was partially supported by the PSC-CUNY award
\# 61319-0039 and NSF grant DMS-0825081. 
T.\ Mountford was supported by SNF grant 200020-115964 and by the C.N.R.S..
 M.\ Zerner's work was partially supported by ERC Starting Grant 208417-NCIRW.  

\bibliographystyle{amsalpha}

\begin{thebibliography}{[Wa01]}

\bibitem[Du05]{durr}{\sc R.\ Durrett}.
\textit{Probability: Theory and Examples.} Duxbury Press, 3rd edition (2005).

\bibitem[Fl07]{Fl07}{\sc M.\ Flury}.
Large deviations and phase transition for random walks in random nonnegative potentials.
\textit{Stoch.\ Proc.\ Appl.} {\bf 117}, 596--612 (2007). 

\bibitem[Fl08]{Fl08}{\sc M.\ Flury}.
Coincidence of Lyapunov exponents for random walks in weak random potentials.
\textit{Ann.\ Probab.} {\bf 36}, No.\ 4, 1528--1583 (2008)

\bibitem[Sz98]{Sz98}{\sc A.-S.\ Sznitman.} \textit{Brownian motion, obstacles and random media.} Springer Monographs in Mathematics. Springer-Verlag, Berlin, 1998. 

\bibitem[Wa01]{Wa01}{\sc Wei-Min Wang}.
Mean field bounds on Lyapunov exponents in $\Z^d$ at the critical energy.
\textit{Probab.\ Theory Relat. Fields} {\bf 119}, No.\ 4, 453--474 (2001).

\bibitem[Wa02]{Wa02}{\sc Wei-Min Wang}.
Mean field upper and lower bounds on Lyapunov exponents.
\textit{Am.\ J.\ Math.} {\bf 124}, No.\ 5, 851--878 (2002). 

\bibitem[Ze98]{Ze98} {\sc M.P.W.\ Zerner}. Directional decay of the Green's function for a random nonnegative
potential on $\Z^d$. \textit{Ann.\ Appl.\ Probab.} {\bf 8}, 246--280 (1998).

\bibitem[Ze00]{Ze00}  {\sc M.P.W.\ Zerner}. Velocity and Lyapounov exponents of some random walks in random environment.
\textit{Ann.\ de l'IHP, Prob.\ et Stat} {\bf 36}, No.\ 6, 737--748 (2000).

\bibitem[Zy09]{Zy09}{\sc N.\ Zygouras}. Lyapounov norms for random walks in low disorder and dimension greater than three.  \textit{Probab.\ Theory Relat.\ Fields} {\bf 143}, No.\ 3-4, 615--642 (2009).


\end{thebibliography}
\vspace*{5mm}
{\sc \small
\begin{tabular}{ll}
Department of Mathematics& \hspace*{20mm}Ecole Polytechnique F\'ed\'erale\\
Baruch College, Box B6-230& \hspace*{20mm}de Lausanne\\
One Bernard Baruch Way&\hspace*{20mm}D\'epartement de math\'ematiques\\
New York, NY 10010, USA&\hspace*{20mm}1015 Lausanne, Switzerland\\
{\verb+elena.kosygina@baruch.cuny.edu+}& \hspace*{20mm}{\verb+thomas.mountford@epfl.ch+}
\end{tabular}\vspace*{2mm}

\begin{center}
\begin{tabular}{l}
Eberhard Karls Universit\"at T\"ubingen\\
Mathematisches Institut\\
Auf der Morgenstelle 10\\
72076 T\"ubingen, Germany\\
{\verb+martin.zerner@uni-tuebingen.de+} 
\end{tabular}
\end{center}
}

\end{document}